\documentclass[11pt]{amsart}

\usepackage{AFBoix2015}

\begin{document}

\author[A.\,F.\,Boix]{Alberto F.\,Boix$^{*}$}
\thanks{$^{*}$Partially supported by Spanish Ministerio de Econom\'ia y Competitividad MTM2016-7881-P}
\address{Department of Economics and Business, Universitat Pompeu Fabra, Jaume I Building, Ramon Trias Fargas 25-27, 08005 Barcelona, Spain.}
\email{alberto.fernandezb@upf.edu}

\author[M.\,Eghbali]{Majid Eghbali$^{**}$}
\thanks{$^{**}$Partially supported by a grant from IPM (No. 93130017)}
\address{School of Mathematics, Institute for Research in Fundamental Sciences (IPM), P. O. Box: 19395-5746,
Tehran-Iran.}
\email{m.eghbali@yahoo.com}

\keywords{Annihilators, local cohomology.}

\subjclass[2010]{Primary 13D45; Secondary 13A35, 13N10}

\title[Annihilators of local cohomology and differential operators]{Annihilators of local cohomology modules and simplicity of rings of differential operators}

\begin{abstract}
One classical topic in the study of local cohomology is whether the non-vanishing of a specific local cohomology module is equivalent to the vanishing of its annihilator; this has been studied by several authors, including Huneke, Koh, Lyubeznik and Lynch. Motivated by questions raised by Lynch and Zhang, the goal of this paper is to provide some new results about this topic, which provide some partial positive answers to these questions. The main technical tool we exploit is the structure of local cohomology as module over rings of differential operators.
\end{abstract}

\maketitle

\section{Introduction}
Let $R$ be a commutative Noetherian ring, and let $I$ be an ideal of $R$. For $i\in\N$ and for an $R$-module $M$, we denote by $H_I^i (M)$ the $i$-th local cohomology module with respect to $I$ (see \cite[1.2.1]{BroSha}). The purpose of this paper is to study annihilators of local cohomology modules and, more precisely, how to deduce their vanishing from the fact that local cohomology modules admit an additional structure as modules over rings of differential operators (in other words, that local cohomology modules may be regarded as solutions of systems of partial differential equations). First of all, we want to review some of the known results concerning the vanishing of these annihilators.

\begin{enumerate}[(i)]

\item Let $R$ be a regular local ring containing a field, and let $I$ be an ideal of $R$. Then, $H_I^i (R)\neq 0$ if and only if $\left(0:_R H_I^i (R)\right)=0$; this result was proved by Huneke and Koh (see \cite[Lemma 2.2]{HunKoh91}) in prime characteristic, and by Lyubeznik (see \cite[Corollary 3.6]{Lyubeznik1993Dmod}) in characteristic zero (see also \cite[Corollary 4.4]{SarriaCallejasCaro2017} for a more general statement in characteristic zero). As pointed out by Lynch (see \cite[Top paragraph of page 543]{Lyn12}), this result remains open for arbitrary regular local rings of mixed characteristic.

\item Let $R$ be a local ring, $I\subseteq R$ an ideal generated by $t$ elements, and set $A_i:=\left(0:_R H_I^i (R)\right)$. Then, $A_0\cdots A_t=0$ (see \cite{Schenzel1982}). In particular, if $I$ is a cohomologically complete interesection ideal (i.e., $H_I^i (R)=0$ for all $i\neq g:=\grade_R (I)$), then $\left(0:_R H_I^g (R)\right)=0$ (see \cite[Corollary 3.2 and Theorem 3.3]{Lyn12}).

\item Let $R$ be a local domain with dimension $d\leq 3$, and let $I$ be an ideal of $R$ with $c:=\cdim (I)$, where $\cdim (I):=\sup\{i\mid H_I^i (R)\neq 0\}$. Then, $\left(0:_R H_I^c (R)\right)=0$ (see \cite[Theorem 4.4]{Lyn12}); this result has been recently extended by Atazadeh, Sedghi and Naghipour for Noetherian domains of arbitrary dimension provided $c=d-1$ (see \cite[Lemma 3.6]{AtSeNa14}).

\item Let $R$ be a complete, Cohen-Macaulay, unique factorization domain of dimension at most $4$, and let $I$ be an ideal of $R$. Then, $H_I^i (R)\neq 0$ if and only if $\left(0:_R H_I^i (R)\right)=0$ (see \cite[Theorem 4.2]{Lyn11}).

\item Let $(R,\mathfrak{m})$ be a local ring of dimension $d$, and let $I$ be an ideal of $R$ such that $H_I^d (R)\neq 0$. Then,
\[
\left(0:_{\hat{R}} H_I^d (R)\right)=\bigcap\{Q\mid\ Q \text{ primary component of } \hat{R},\ \dim \hat{R}/Q=d,\ \dim \hat{R}/I+Q=0\}.
\]
If $R$ is unmixed and $\sqrt{I}=\mathfrak{m}$, then $\left(0:_R
H_I^d (R)\right)=0$; the converse holds provided $R$ is, in
addition, complete (see \cite[Theorem 2.4 and Corollary 2.5]{Lyn12} and \cite[Theorem
4.2(a)]{Egh-Sch12}).

\item It is worth noting here that there are also several nice results concerning the annihilators of local cohomology modules $H_I^i (M)$, where $M$ is a finitely generated $R$-module; the interested reader may like to consult \cite{BaAzGh12}, \cite{AtSeNa14} and \cite{AtSeNa15} for further details.

\end{enumerate}
One of the motivations of this paper was to understand a bit better the structure of annihilators of local cohomology modules in mixed characteristic; in particular, in this mixed characteristic setting we have as guide the following:

\begin{quo}\label{starting question}
Let $A$ be a Dedekind domain such that, for each height one prime ideal $\mathfrak{p}$ of $A$, $A_{\mathfrak{p}}$ has mixed characteristic. Moreover, let $R$ be a regular commutative domain which, in addition, is an $A$-algebra, and $I\subset R$ an ideal. Is it true that $H_I^i (R)\neq 0$ if and only if $\left(0:_R H_I^i (R)\right)=0$?
\end{quo}
The reader will easily note that Question \ref{starting question} is the natural analogue in mixed characteristic of the
above Huneke-Koh-Lyubeznik Theorem; it turns out that one of the main results of our paper (see Theorem \ref{main result of the paper}) provides a partial positive answer to Question \ref{starting question}.

On the other hand, during the Algebra week of the Mathematics Research Communities held at Utah in 2015, W.\,Zhang raised the following:

\begin{quo}\label{Wenliang question}
Let $R$ be a commutative Noetherian ring, and let $I$ be any ideal of $R$. What conditions one has to require either to $R$ or $I$ to ensure that $H_I^i (R)\neq 0$ if and only if $\left(0:_R H_I^i (R)\right)=0$?
\end{quo}
As pointed out by Lynch in \cite[page 543]{Lyn12}, one has to require at least that $R$ is reduced. As we have seen so far in this Introduction, Zhang's question has a positive answer provided $R$ is a regular local ring containing a field, or $R$ is a complete, Cohen-Macaulay, unique factorization domain of dimension at most $4$. On the other hand, assume for a while that $R$ has prime characteristic $p$; it is known that Zhang's question has a positive answer provided $R$ is a strongly $F$-regular domain; in Theorem \ref{answer to Zhang question} we partially recover this result. This also confirms the implicit prediction made by Lynch in \cite[Question 6]{Lyn11}, where she asked whether prime characteristic methods might be useful to deduce further properties of annihilators of local cohomology modules.

Now, we provide a more detailed overview of the contents of this paper for the reader's benefit. After reviewing in Section \ref{preliminaries} some preliminary material, we give in Section \ref{annihilators of regular domains}, mostly in the spirit of \cite{Lyu2000id}, \cite{Lyu2000fp} and \cite{Lyubeznik2011}, a characteristic-free proof of the Huneke-Koh-Lyubeznik result (see Theorem \ref{annihilators over regular rings: main result}) building upon the simplicity of the ring of differential operators over certain regular domains (see Theorem \ref{simplicity of differential operators of smooth varieties}); as first application, we recover a criterion (see Corollary \ref{reminders of Grothendieck false conjecture}) which ensures the vanishing of a local cohomology module, which was obtained by Huneke and Koh in prime characteristic, and by Lyubeznik in characteristic zero (see \cite[Corollary 3.6 (e)]{Lyubeznik1993Dmod}). Section \ref{annihilators in mixed characteristic} contains one of the main results of this paper (see Theorem \ref{main result of the paper}); namely, a partial positive answer to Question \ref{starting question}. On the other hand, in Section \ref{annihilators in prime characteristic} we partially recover the affirmative answer to Zhang's Question \ref{Wenliang question} (see Theorem \ref{answer to Zhang question}) in case of a strongly $F$-regular domain. In Section \ref{annihilators via flat endomorphisms}, we provide another proof of the Huneke-Koh result, using as main tool the existence of a flat endomorphism over certain commutative rings (see Theorem \ref{flat endomorphisms}); on the other hand, Section \ref{annihilators of normal rings} is essentially devoted to the study of the vanishing of these annihilators over normal rings. Among the results obtained in this section, the main one (see Theorem \ref{main result of section on normal rings}) provides a mild generalization of \cite[Lemma 3.6]{AtSeNa14} for Noetherian reduced normal rings of arbitrary dimension. The main motivation for us to write down Sections \ref{annihilators via flat endomorphisms} and \ref{annihilators of normal rings} was to give some answers to the following question, raised by Lynch in her thesis (see \cite[Page 42, Question 5]{Lyn11}).

\begin{quo}[Lynch]\label{annihilators over normal and C.M. rings}
Let $(R,\fm ,K)$ be a Cohen-Macaulay, local normal domain, and let $I\subset R$ be any ideal. Is it true that $H_I^i (R)\neq 0$ if and only if $\left(0:_R H_I^i (R)\right)=0$?
\end{quo}
It is worth noting that Question \ref{annihilators over normal and C.M. rings} is quite natural, because it is known that any strongly $F$-regular ring is Cohen-Macaulay and normal (see \cite[(5.5) (d)]{HochsterHuneke1994} and \cite[(4.9)]{HoHun1990}).

Finally, in Section \ref{annihilators and Macaulay rings} we provide some results concerning another question raised by Lynch in her thesis (see Question \ref{Lynch question on Macaulay rings} and Theorem \ref{some answer to Lynch question on Macaulay rings}).


This paper recovers and extends the results obtained by the second named author in \cite{Eghbali15}.

\section{Preliminaries}\label{preliminaries}
The goal of this section is to single out some technical facts which we use in the next sections; it is worth noting that, albeit all of the results contained in this section are well known, we would like to write them here for the sake of completeness.

Throughout this section $A$ is a commutative ring, $B$ is a (not necessarily commutative) ring, and $\xymatrix@1{A\ar[r]^-{\phi}& B}$ is a ring homomorphism. Given a left $B$-module $M$, we look at $M$ as left $A$-module via restriction of scalars under $\phi$; that is, $M$ will be regarded as left $A$-module with multiplication given by $a\cdot m:= \phi (a)m$, where $a\in A$ and $m\in M$. 

\begin{lm}\label{simple imply injective}
The following assertions hold.

\begin{enumerate}[(i)]

\item The annihilator $(0:_B M):=\{b\in B\mid\ bm=0\text{ for all }m\in M\}$ is a two-sided ideal of $B$.


\item If $(0:_B M)=0$, then $(0:_A M)\subseteq\ker (\phi).$

\end{enumerate}
\end{lm}

\begin{proof}
First of all, we check that $J:=(0:_B M)$ is a two-sided ideal of $B$; indeed, given $x\in J$ and $y\in B$, we only have to note that
\[
(xy)M=x(yM)\subseteq xM=0\text{ and }(yx)M=y(xM)\subseteq xM=0.
\]
The above equalities show that $xy\in J$ and $yx\in J$ respectively, hence $(0:_B M)$ is a two-sided ideal of $B$; therefore, part (i) holds.

Finally, part (ii) can be proved in the following way; let $a\in (0:_A M),$ so $a\cdot m=0$ for any
$m\in M,$ which is equivalent to say $\phi (a)m=0$ for any
$m\in M;$ in this way, if $(0:_B M)=0,$ then $\phi (a)=0$ and
hence $a\in\ker (\phi),$ as claimed.
\end{proof}
Another technical fact we plan to use later on (see proof of Theorem \ref{annihilators over regular rings: main result}) is the following:

\begin{lm}\label{annihilators under faithfully flat base change}
Let $\xymatrix@1{A\ar[r]& A'}$ be a faithfully flat homomorphism of commutative Noetherian rings, let $I$ be an ideal of $A$, and let $J'$ be an ideal of $A'$. Moreover, suppose that $\left(0:_{A'} H_{IA'}^i (A')\right)\subseteq J'$. Then, $\left(0:_A H_I^i (A)\right)\subseteq J'\cap A$; in particular, if $\left(0:_{A'} H_{IA'}^i (A')\right)=0$ then $\left(0:_A H_I^i (A)\right)=0$.
\end{lm}

\begin{proof}
The result follows from the following ascending chain of inclussions:
\[
\left(0:_A H_I^i (A)\right)=\left(0:_A H_I^i (A)\right)A'\cap A\subseteq\left(0:_{A'} H_I^i (A)\otimes_A A'\right)\cap A=\left(0:_{A'} H_{IA'}^i (A')\right)\cap A\subseteq J'\cap A.
\]
Indeed, starting from the left; the first equality stems from the fact that faithfully flat ring maps are, in particular, contracted. On the other hand, since $\xymatrix@1{A\ar[r]& A'}$ is faithfully flat it is, in particular, pure, hence there is an injection $\xymatrix@1{H_I^i (A)\ar@{^{(}->}[r]& H_I^i (A)\otimes_A A',}$ and therefore the inclusion $\left(0:_A H_I^i (A)\right)A'\cap A\subseteq\left(0:_{A'} H_I^i (A)\otimes_A A'\right)\cap A$ follows directly from this fact. Finally, the equality $\left(0:_{A'} H_I^i (A)\otimes_A A'\right)\cap A=\left(0:_{A'} H_{IA'}^i (A')\right)\cap A$ boils down into the natural isomorphism
\[
H_I^i (A)\otimes_A A'\cong H_{I A'}^i (A')
\]
given by flat base change (see \cite[4.3.2]{BroSha}); the proof is therefore completed.
\end{proof}
Once again, remember that $B$ is a (not necessarily commutative) ring; the below well-known notion will play a key role throughout this paper.

\begin{df}\label{definition of simple rings}
It is said that $B$ is \emph{simple} provided its unique two-sided ideals are $0$ and $B$.
\end{df}

As a direct application of Lemma \ref{simple imply injective}, we obtain the following fact, which may be regarded as the main result of this section.

\begin{teo}\label{simple implies faith}
Let $A$ be a commutative ring, let $B$ be a (not necessarily commutative) ring such that $A\subseteq B$ and let $M$ be a left $B$-module. If $B$ is simple, then $\left(0:_A M\right)=0$.
\end{teo}
Next Remark will play some role later on (see Proof of Theorem \ref{simplicity of differential operators of smooth varieties}).

\begin{rk}\label{pullback of ideals under nice projections}
Let $A\subseteq A'$ be an integral extension of commutative domains, and let $I'$ be an ideal of $A'$. We claim that if $I'\neq 0$, then $I'\cap A\neq 0$; indeed, since $I'\neq 0$ pick $a'\in I'$ non-zero. Since $A\subseteq A'$ is integral, $a'$ verifies an equation of the form
\[
X^n +a_{n-1} X^{n-1}+\ldots +a_1 X+a_0=0,
\]
for some $a_j\in A$ ($0\leq j\leq n-1$). Since $A'$ is a domain, $a_0\neq 0$ (otherwise, the previous equation would imply that $a'$ would be a zero-divisor of $A'$) and therefore one has
\[
a_0= -a_1 a'-a_2 (a')^2-\ldots -a_{n-1} (a')^{n-1}-(a')^n\in I'.
\]
This shows that $I'\cap A\neq 0$, as required.
\end{rk}
Although next notion is well-known, we want to write it down here to avoid any misunderstanding.

\begin{df}\label{definition of regular rings}
Let $(R,\fm ,K)$ be a commutative Noetherian local ring of dimension $d$; if there are elements $x_1,\ldots ,x_d\in\fm$ such that $\fm =(x_1,\ldots ,x_d)$, then it is said that $R$ is a \emph{regular local ring} and that $x_1,\ldots ,x_d$ forms a \emph{regular system of parameters}.

More generally, given a commutative Noetherian (not necessarily local) ring $A$, it is said that $A$ is \emph{regular} if, for any maximal ideal $\fm$ of $A$, $A_{\fm}$ is a regular local ring.
\end{df}
Our next aim is to review the notion of smooth algebras, because it will play a key role later on in this paper (see Theorem \ref{main result of the paper}).

\begin{df}\label{geometrically regular rings}
Let $K$ be a field, and let $B$ be a Noetherian $K$-algebra. It is said that $B$ is \emph{geometrically regular} if, for any finite field extension $F|K$, the Noetherian ring $B\otimes_K F$ is regular.
\end{df}

\begin{df}
Let $A$ be a commutative ring, and let $R$ be a commutative $A$-algebra. It is said that $R$ is \emph{smooth over $A$} provided $R$ is a finitely presented, flat $A$-algebra, and for each prime ideal $\mathfrak{p}$ of $A$, the fiber $R_{\mathfrak{p}}/\mathfrak{p}R_{\mathfrak{p}}$ is geometrically regular over $A_{\mathfrak{p}}/\mathfrak{p}A_{\mathfrak{p}}$.
\end{df}
Our main reason for considering smooth algebras is the following important property, which we shall exploit later on in this paper (see Proofs of Theorems \ref{annihilators in mixed characteristic: first result} and \ref{main result of the paper}); the interested reader may like to consult \cite[(2.1) and Lemma 2.1]{BBLSZ14} for details.

\begin{lm}\label{smooth algebras and differential operators under base change}
Let $A$ be a commutative ring, let $R$ be either a polynomial ring over $A$, a formal power series ring over $A$, or a smooth $A$-algebra. Moreover, let $D_A (R)$ be the ring of $A$-linear differential operators on $R$. Then, for any $A$-algebra $B$ one has
\[
D_A (R)\otimes_A B\cong D_B \left(R\otimes_A B\right).
\]
In particular, for any element $a\in A$, one has $D_A (R)/a D_A (R)\cong D_{A/aA} \left(R/aR\right).$
\end{lm}

\section{Annihilators of local cohomology modules in certain regular domains}\label{annihilators of regular domains}
As explained in the Introduction of this manuscript, given $R$ a regular local ring containing a field, and $I\subset R$ any ideal, it is known that $H_I^i (R)\neq 0$ if and only if $\left(0:_R H_I^i (R)\right)=0$; this result was proved by Huneke and Koh in prime characteristic (see \cite[Lemma 2.2]{HunKoh91}) building upon the flatness of the Frobenius endomorphism, and by Lyubeznik in characteristic zero (see \cite[Corollary 3.6]{Lyubeznik1993Dmod}) carrying over the $D$-module structure naturally enhanced to local cohomology modules. Following the spirit of \cite{Lyu2000id}, \cite{Lyu2000fp} and \cite{Lyubeznik2011}, the purpose of this section is to provide a characteristic-free proof of the forementioned Huneke-Koh-Lyubeznik result.

Throughout this section, $R$ will denote a commutative Noetherian ring containing a field $K$, and let $D_K (R)$ be the ring of $K$-linear differential operators (see, among other places, \cite[Definition 17.1]{Twentyfourhours}). Now, we review some known results which we need later on. The first one is the simplicity of the ring of differential operators of a polynomial ring; the interested reader may consult \cite[Theorem 2.1.9]{Traves1998} for details.

\begin{prop}\label{simplicity of differential operators of affine space}
$D_K \left( K[x_1,\ldots ,x_n]\right)$ is a simple ring.
\end{prop}

The second auxiliary result we have to review is actually a technical Lemma proved by Smith in \cite[page 173]{SPSmith1986}.

\begin{lm}\label{pullback of ideals in rings of differential operators}
If $J$ is a non-zero two-sided ideal of $D_K (R)$, then $J\cap R\neq 0$.
\end{lm}
The last technical result we need to review is the simplicity of the ring of differential operators of formal and convergent power series rings; because of the lack of a reference when $n\geq 2$ (in case of $n=1$ see \cite[Proposition 1.3.3]{Sabbah1993}) we provide a proof; in the below result, $K\{x_1,\ldots ,x_n\}$ denotes the convergent power series ring with coefficients in a complete valued field $K$; that is, $K$ is a field endowed with a valuation such that every Cauchy sequence of elements in $K$ converges in $K$.

\begin{prop}\label{simplicity of differential operators of affinoid space}
$D_K \left( K[\![x_1,\ldots ,x_n]\!]\right)$ and $D_K \left( K\{x_1,\ldots ,x_n\}\right)$ (in this case, $K$ is a complete valued field) are simple rings.
\end{prop}

\begin{proof}
Let $J$ be a non-zero two-sided ideal of $D_K \left( R\right)$, where $R$ is either $K[\![x_1,\ldots ,x_n]\!]$ or $K\{x_1,\ldots ,x_n\}$; by Lemma \ref{pullback of ideals in rings of differential operators}, $J$ contains a non-zero element of $R$, namely $f$. By Weierstrass Preparation Theorem, one can write $f=up$, where $u$ is a unit of $R$ and $p$ is a polynomial; this implies that $p=f/u\in J$, so $J\cap K[x_1,\ldots ,x_n]\neq 0$ and, in particular, $J\cap D_K \left( K[x_1,\ldots ,x_n]\right)\neq 0$. But $D_K \left( K[x_1,\ldots ,x_n]\right)$ is simple, whence $1\in J$; the proof is therefore completed.
\end{proof}

The last technical definition we need to review is the following:

\begin{df}\label{definition of affinoid algebras}
Let $K$ be a complete valued field. A $K$-\emph{affinoid algebra} is a $K$-algebra $A$ admitting an isomorphism $A\cong K\{x_1,\ldots ,x_n\}/I$ for some ideal $I\subseteq K\{x_1,\ldots ,x_n\} .$
\end{df}
In this way, we are ready to prove the first main result of this section; as the reader will easily note, the proof of this result is just a slight variation of \cite[Proof of Proposition 3.4]{SPSmith1986}; anyway, we want to write it down for the convenience of the reader. 

\begin{teo}\label{simplicity of differential operators of smooth varieties}
Let $K$ be any field, and let $R$ be a commutative Noetherian ring; moreover, suppose that either $R$ is a regular ring, affine over $K$, or $R$ is a formal power series ring over $K$, or that $R$ is a regular ring, affinoid over $K$ and $K$ is a complete valued field. Then, $D_K (R)$ is simple.
\end{teo}

\begin{proof}
The case when $R$ is a formal power series ring is covered by Proposition \ref{simplicity of differential operators of affinoid space}, so hereafter in this proof we assume that $R$ is a regular ring, either affine or affinoid over $K$ (in the affinoid case, we have to further assume that $K$ is a complete valued field).

Under our assumptions, it is known that, for any maximal ideal $\fm$ of $R$, $D_K (R_{\fm})=R_{\fm}\otimes_R D_K (R)$; in this way, it is enough to show that, for any maximal ideal $\fm$ of $R$, $D_K (R_{\fm})$ is a simple ring. Indeed, let $I$ be a proper, two-sided ideal of $D_K (R)$; on one hand, Lemma \ref{pullback of ideals in rings of differential operators} ensures that $I':=I\cap R\neq 0$. On the other hand, since $I\neq D_K (R)$, $I$ does not contain any unit of $D_K (R)$; this implies that $I'$ does not contain any unit of $R$. Thus, $I'$ is a proper ideal of $R$, and therefore $I' R_{\fm}$ is also a proper ideal of $R_{\fm}$ for at least one maximal ideal $\fm$ of $R$; this implies that $D_K (R_{\fm}) (I'R_{\fm})$ is a proper, two-sided ideal of $D_K (R_{\fm})$. Summing up, we have seen that, if $D_K (R)$ is not simple, then $D_K (R_{\fm})$ is also not simple for at least one maximal ideal $\fm$ of $R$.

Therefore, from now on we plan to show the simplicity of $D_K (R_{\fm})$ for any maximal ideal $\fm$; indeed, since $R_{\fm}$ is a regular local ring, we can pick elements $x_1,\ldots ,x_n\in R_{\fm}$ ($n=\dim (R_{\fm})$) such that $x_1,\ldots ,x_n$ forms a regular system of parameters for $R_{\fm}$, with the property that $\Omega_{R_{\fm}}$ (the module of K\"{a}hler differentials) is free on $dx_1,\ldots ,dx_n$. In this setting, it is known (see \cite[Section 16]{EGAIV}) that $D_K (R_{\fm})$ is generated by $R_{\fm}$ and a set of differential operators
\[
\{ D_{\alpha}\mid\ \alpha =(a_1,\ldots ,a_n),\ 0\leq a_j<\infty\text{ for all }j\}
\]
which satisfy
\[
D_{\alpha} (\mathbf{x}^{\beta})=\binom{\beta}{\alpha} \mathbf{x}^{\beta-\alpha},
\]
where $\beta =(b_1,\ldots ,b_n)$, $\mathbf{x}^{\beta-\alpha}=x_1^{b_1-a_1}\cdots x_n^{b_n-a_n}$, and
\[
\binom{\beta}{\alpha}:=\prod_{j=1}^n \binom{b_j}{a_j}.
\]
Now, the key point one has to take into account is that $D_K (R_{\fm})$ contains $D_K (K[x_1,\ldots ,x_n])$ (respectively, $D_K (K\{x_1,\ldots ,x_n\})$ in the affinoid case); namely, the subalgebra generated by the polynomial ring $K[x_1,\ldots ,x_n]$ (respectively, the convergent power series ring $K\{x_1,\ldots ,x_n\})$ and all the differential operators $D_{\alpha}$'s.

After all the above preliminaries, we are in position to prove that $D_K (R_{\fm})$ is simple; indeed, let $J$ be any two-sided ideal of $D_K (R_{\fm})$. Lemma \ref{pullback of ideals in rings of differential operators} ensures that $J\cap R_{\fm}\neq 0$, and consequently, since the inclusion $K[x_1,\ldots ,x_n]\subseteq R_{\fm}$ (respectively, $K\{x_1,\ldots ,x_n\}\subseteq R_{\fm}$ in the affinoid case) is an integral extension of normal domains, Remark \ref{pullback of ideals under nice projections} implies that $J\cap K[x_1,\ldots ,x_n]\neq 0$ (respectively, $J\cap K\{x_1,\ldots ,x_n\}\neq 0$); in particular, $J\cap D_K (K[x_1,\ldots ,x_n])\neq 0$ (respectively, $J\cap D_K (K\{x_1,\ldots ,x_n\})\neq 0$). However, $D_K (K[x_1,\ldots ,x_n])$ (respectively, $D_K (K\{x_1,\ldots ,x_n\})$) is a simple ring (see Proposition \ref{simplicity of differential operators of affine space}, respectively Proposition \ref{simplicity of differential operators of affinoid space}), hence $1\in J$. This fact guarantees that $D_K (R_{\fm})$ is a simple ring, just what we finally wanted to prove.
\end{proof}
The promised characteristic-free proof of the Huneke-Koh-Lyubeznik result turns out to be a direct consequence of Theorem \ref{simple implies faith} (with $A=R$, $B=D_K (R)$ and $M=H_I^j (R)$) and Theorem \ref{simplicity of differential operators of smooth varieties}; this is the second main result of this section.

\begin{teo}\label{annihilators over regular rings: main result}
Let $K$ be any field, let $R$ be a commutative Noetherian ring, and let $I\subset R$ be any ideal. Moreover, suppose that $R$ is either a regular ring, affine or affinoid over $K$, or that $R$ is a regular local ring with $K$ as residue field. Then, $H_I^i (R)\neq 0$ if and only if $\left(0:_R H_I^i (R)\right)=0$.
\end{teo}

\begin{proof}
On one hand, the case where $R$ is regular and either affine or affinoid over $K$ is a direct consequence of Theorem \ref{simple implies faith} (with $A=R$, $B=D_K (R)$ and $M=H_I^j (R)$) and Theorem \ref{simplicity of differential operators of smooth varieties}; on the other hand, assume now that $R$ is a regular local ring having $K$ as residue field. By Lemma \ref{annihilators under faithfully flat base change}, we can further assume that $R$ is complete. However, in this case $R$ is a formal power series ring over $K$, and therefore the result follows once again combining Theorem \ref{simple implies faith} (with $A=R$, $B=D_K (R)$ and $M=H_I^j (R)$) and Theorem \ref{simplicity of differential operators of smooth varieties}.
\end{proof}

As announced in the Introduction of this manuscript, we recover the following criterion which ensures the vanishing of a local cohomology module (cf.\,\cite[Corollary 3.6 (e)]{Lyubeznik1993Dmod}).

\begin{cor}\label{reminders of Grothendieck false conjecture}
Let $K$ be any field, let $R$ be a commutative Noetherian ring, and let $I\subset R$ be an ideal. Suppose that either $R$ is regular, affine or affinoid over $K$, or that $R$ is a regular local ring having $K$ as residue field. Then, given $j$ an integer bigger than the height of all the minimal primes of $I$, the following statements are equivalent:

\begin{enumerate}[(i)]

\item $H_I^j (R)=0$.

\item $\left(0:_R H_I^j (R)\right)\neq 0$.

\item $\Hom_R \left(R/I,H_I^j (R)\right)$ is a finitely generated $R$-module.

\end{enumerate}
\end{cor}

\begin{proof}
Once again, the equivalence between (i) and (ii) follows directly from Theorem \ref{annihilators over regular rings: main result}. Since (i) implies obviously (iii), it is enough to show that (iii) implies (i), which can be proved mutatis mutandis as in \cite[Theorem 2.3]{HunKoh91}; the proof is therefore completed.
\end{proof}

\section{Annihilators of local cohomology modules in mixed characteristic}\label{annihilators in mixed characteristic}
The study of annihilators of local cohomology modules in mixed characteristic is very often a hard matter (see \cite{RobertsSinghSrinivas2007} as a taste); in \cite[Top paragraph of page 543]{Lyn12}, Lynch asked implicitly whether the Huneke-Koh-Lyubeznik result is also true for arbitrary regular local rings of mixed characteristic. The goal of this section is to show that, when $R$ is a smooth $\Z$-algebra verifying some additional condition (see Definition \ref{the star assumption} and Theorem \ref{main result of the paper}), Lynch question has an affirmative answer (see Theorem \ref{main result of the paper}).

Next statement is the first main result of this section; as the reader will easily note, our proof follows so closely the steps carried out in \cite[Theorem 4.1]{BBLSZ14}.

\begin{teo}\label{annihilators in mixed characteristic: first result}
Let $(V,uV)$ be a discrete valuation ring of mixed characteristic with uniformizer $u$, let $R$ be a $V$-algebra that is either smooth over $V$, or a formal power series ring over $V$ and let $D_V (R)$ be the ring of $V$-linear differential operators. Moreover, let $I$ be an ideal of $R$ and fix $i\in\N$. Then:
\begin{enumerate}[(i)]

\item If the multiplication by $u$ $\xymatrix@1{H_I^i (R)\ar[r]^-{\cdot u}& H_I^i (R)}$ is not injective, then $\left(0:_R H_I^i (R)\right)\subseteq uR$, and equality holds if and only if this multiplication is the zero map.

\item If the multiplication by $u$ $\xymatrix@1{H_I^i (R)\ar[r]^-{\cdot u}& H_I^i (R)}$ is not surjective, then $\left(0:_R H_I^i (R)\right)\subseteq uR$.

\end{enumerate}
\end{teo}

\begin{proof}
Firstly suppose that the multiplication by $u$ $\xymatrix@1{H_I^i (R)\ar[r]^-{\cdot u}& H_I^i (R)}$ is not injective and set $M$ as its kernel, which is a non-zero left $D_V (R)$-module. Since $M$ is annihilated by $u$, $M$ can also be regarded as left module over the ring $D_V (R)/u D_V (R)$; moreover, using Lemma \ref{smooth algebras and differential operators under base change} it follows that
\[
D_V (R)/u D_V (R)\cong D_{V/uV} \left(R/uR\right).
\]
In addition, one deduces from Theorem \ref{annihilators over regular rings: main result} that $D_{V/uV} \left(R/uR\right)$ is a simple ring (indeed, since $R$ is either smooth or a formal power series ring over $V$, it follows that $R/uR$ is either smooth or a formal power series ring over the field $V/uV$, because the property of being smooth is stable under base change). Summing up, $M$ is a non-zero left $D_{V/uV} \left(R/uR\right)$-module, hence Theorem \ref{simple implies faith} implies that $\left(0:_{R/uR} M\right)=0$ and therefore $uR=\left(0:_R M\right)$. Finally, since $\left(0:_R H_I^i (R)\right)\subseteq\left(0:_R M\right),$ one concludes that $\left(0:_R H_I^i (R)\right)\subseteq uR$, hence part (i) holds.

On the other hand, assume that the multiplication by $u$ $\xymatrix@1{H_I^i (R)\ar[r]^-{\cdot u}& H_I^i (R)}$ is not surjective, and set $C$ as the cokernel of this multiplication; by the same argument used before with $M$, one can ensure that $\left(0:_R C\right)=uR$, and therefore $\left(0:_R H_I^i (R)\right)\subseteq\left(0:_R C\right)=uR$, just what we finally wanted to show.
\end{proof}
The following notion will play a crucial role very soon (see Proof of Theorem \ref{main result of the paper}).

\begin{df}\label{the star assumption}
Let $R$ be a commutative $\Z$-algebra. We say that $R$ verifies the \emph{star assumption} if, for any infinite set of prime integers $\{p_i\}_{i\in I}$ (that is, $I$ is an infinite set of indexs such that, for each $i\in I$, $p_i$ is a prime integer) one has that
\[
\bigcap_{i\in I} p_i R=(0).
\]
\end{df}

\begin{ex}
We think that this is the correct place to show some examples of rings verifying the star assumption.

\begin{enumerate}[(a)]

\item $\Z$ clearly verifies the star assumption, because any integer has only a finite number of prime divisors; more generally, given a number field $K$, its ring of integers $\mathcal{O}_K$ also satisfies the star assumption.

\item If $R$ verifies the star assumption, then $R[X]$ and $R[\![X]\!]$ also verify the star assumption. Indeed, set $A$ as $R[X]$ or $R[\![X]\!]$. In both cases, notice that
\[
\bigcap_{i\in I} p_i A=\bigcap_{i\in I} \left(p_i R\right)A=\left(\bigcap_{i\in I} p_i R\right)A=(0).
\]
In particular, $\Z[X_1,\ldots ,X_d]$ and $\Z[\![X_1,\ldots ,X_d]\!]$ verify the star assumption.

\item More generally, suppose that $R$ verifies the star assumption, and let $A$ be an $\cap$-flat $R$-algebra; that is (see \cite[Page 41]{HochsterHuneke1994}), that $A$ is a flat $R$-algebra and that, for each (not necessarily finite) family of finitely generated $R$-modules $\{M_j\}_{j\in J}$, the natural map
\[
A\otimes_R \left(\bigcap_{j\in J} M_j\right)\longrightarrow\bigcap_{j\in J} \left(A\otimes_R M_j\right)
\]
is an isomorphism; in this case, $A$ also satisfies the star assumption, with exactly the same proof given in part (b).

\end{enumerate}
\end{ex}

The below statement turns out to be the main result of this paper; namely:

\begin{teo}\label{main result of the paper}
Let $R$ be either a smooth $\Z$-algebra, or a formal power series ring over $\Z$, let $I$ be an ideal of $R$, let $p$ be a prime number, and fix $i\in\N$. Then:
\begin{enumerate}[(i)]

\item If the multiplication by $p$ $\xymatrix@1{H_I^i (R)\ar[r]^-{\cdot p}& H_I^i (R)}$ is not injective, then $\left(0:_R H_I^i (R)\right)\subseteq pR$, and equality holds if and only if this multiplication is the zero map.

\item If the multiplication by $p$ $\xymatrix@1{H_I^i (R)\ar[r]^-{\cdot p}& H_I^i (R)}$ is not surjective, then $\left(0:_R H_I^i (R)\right)\subseteq pR$.

\item If, in addition, $R$ verifies the star assumption, and there are infinitely many prime integers $p$ such that multiplication by $p$ on $H_I^i (R)$ is not surjective, then $H_I^i (R)\neq 0$ if and only if $\left(0:_R H_I^i (R)\right)=0$.

\end{enumerate}

\end{teo}

\begin{proof}
Firstly suppose that the multiplication by $p$ $\xymatrix@1{H_I^i (R)\ar[r]^-{\cdot p}& H_I^i (R)}$ is not injective and set $M$ as its kernel, which is a non-zero left $D_{\Z} (R)$-module. Since $M$ is annihilated by $p$, $M$ can also be regarded as left module over the ring $D_{\Z} (R)/p D_{\Z} (R)$; moreover, using Lemma \ref{smooth algebras and differential operators under base change} it follows that
\[
D_{\Z} (R)/p D_{\Z} (R)\cong D_{\mathbb{F}_p} \left(R/pR\right).
\]
In addition, one deduces from Theorem \ref{annihilators over regular rings: main result} that $D_{\mathbb{F}_p} \left(R/pR\right)$ is a simple ring (indeed, since $R$ is either smooth or a formal power series ring over $\Z$, it follows that $R/pR$ is either smooth or a formal power series ring over the field $\Z/p\Z$, because the property of being smooth is stable under base change). Summing up, $M$ is a non-zero left $D_{\mathbb{F}_p} \left(R/pR\right)$-module, hence Theorem \ref{simple implies faith} implies that $\left(0:_{R/pR} M\right)=0$ and therefore $pR=\left(0:_R M\right)$. Finally, since $\left(0:_R H_I^i (R)\right)\subseteq\left(0:_R M\right),$ one concludes that $\left(0:_R H_I^i (R)\right)\subseteq pR$, hence part (i) holds.

On the other hand, assume that the multiplication by $p$ $\xymatrix@1{H_I^i (R)\ar[r]^-{\cdot p}& H_I^i (R)}$ is not surjective, and set $C$ as the cokernel of this multiplication; by the same argument used before with $M$, one can ensure that $\left(0:_R C\right)=pR$, and therefore $\left(0:_R H_I^i (R)\right)\subseteq\left(0:_R C\right)=pR$, hence part (ii) holds too.

Finally, we prove part (iii); indeed, it is known that all but finitely many prime integers are nonzerodivisors on $H_I^i (R)$
(see \cite[Theorem 3.1]{BBLSZ14}). Therefore, part (ii) implies that
\[
\left(0:_R H_I^i (R)\right)\subseteq\bigcap pR,
\]
where the previous intersection runs over all the prime integers which are nonzerodivisors on $H_I^i (R)$ such that, for any prime integer $p$ in this set, the multiplication by $p$ on $H_I^i (R)$ is not surjective. Since, by assumption, this is an infinite intersection and $R$ verifies the star assumption, this intersection must be $0$, hence $\left(0:_R H_I^i (R)\right)=0$, just what we finally wanted to prove.
\end{proof}
The final goal of this section is to show that the assumptions required in part (iii) of Theorem \ref{main result of the paper} are satisfied in some concrete examples; before doing so, we prove a couple of technical lemmas.

\begin{lm}\label{cokernel of multiplication by p}
Let $R$ be a commutative Noetherian ring, let $I$ be an ideal of $R$, let $p$ be a prime number which is a nonzerodivisor on $R$, and fix $i\in\N$; moreover, assume that multiplication by $p$ on $H_I^{i+1} (R)$ is injective. Then, one has that $\Coker (\xymatrix@1{H_I^i (R)\ar[r]^-{\cdot p}& H_I^i (R)})=H_I^i (R/pR).$
\end{lm}

\begin{proof}
The result follows directly from the exactness of
\[
\xymatrix{H_I^i (R)\ar[r]^-{\cdot p}& H_I^i (R)\ar[r]& H_I^i (R/pR)\ar[r]& H_I^{i+1} (R)\ar[r]^-{\cdot p}& H_I^{i+1} (R),}
\]
which is just the one induced by the short exact sequence $\xymatrix@1{0\ar[r]& R\ar[r]^-{\cdot p}& R\ar[r]& R/pR\ar[r]& 0.}$
\end{proof}

\begin{lm}\label{cokernel of multiplication by p: second round}
Let $R$ be either a smooth $\Z$-algebra, or a formal power series ring over $\Z$, let $I$ be an ideal of $R$, let $p$ be a prime number, and fix $i\in\N$. Then, all but finitely many prime integers satisfy that $\Coker (\xymatrix@1{H_I^i (R)\ar[r]^-{\cdot p}& H_I^i (R)})=H_I^i (R/pR).$
\end{lm}

\begin{proof}
By \cite[Theorem 3.1 (2)]{BBLSZ14}, all but finitely many prime integers are nonzerodivisors on $H_I^{i+1} (R),$ and therefore the result follows immediately from Lemma \ref{cokernel of multiplication by p}.
\end{proof}
Now, we are finally in position to show that the assumptions required in part (iii) of Theorem \ref{main result of the paper} are satisfied in some specific situations.

\begin{ex}\label{smooth elliptic curve example}
Let $E$ be a smooth elliptic curve in $\mathbb{P}_{\Q}^2$, and let $I$ be an ideal of $R=\Z [x_0,\ldots ,x_{3n+2}]$ defining the Segre embedding $E\times\mathbb{P}_{\Q}^n\subset\mathbb{P}_{\Q}^{3n+2};$ it is known (see \cite[Example 22.8]{Twentyfourhours}) that the height of $I$ is $2n+1$, so $H_I^{2n+1} (R)\neq 0$ and $H_I^{2n+1} (R/pR)\neq 0$ for any prime integer $p$. In this way, Lemma \ref{cokernel of multiplication by p: second round} ensures that $\xymatrix@1{H_I^{2n+1} (R)\ar[r]^-{\cdot p}& H_I^{2n+1} (R)}$ is not surjective for all but finitely many prime integers $p$; hence Theorem \ref{main result of the paper} implies that $(0:_R H_I^{2n+1} (R))=0.$

On the other hand, it is also known (see \cite[Example 22.8]{Twentyfourhours}) that there are infinitely many prime integers $p$ for which $H_I^{3n+1} (R/pR)\neq 0$; thus, combining this fact joint with Lemma \ref{cokernel of multiplication by p: second round} we obtain that $\xymatrix@1{H_I^{3n+1} (R)\ar[r]^-{\cdot p}& H_I^{3n+1} (R)}$ is not surjective for an infinite set of prime integers, and therefore Theorem \ref{main result of the paper} guarantee that $(0:_R H_I^{3n+1} (R))=0.$
\end{ex}

\section{Annihilators of local cohomology modules in prime characteristic}\label{annihilators in prime characteristic}
Given a commutative ring $R$ of prime characteristic $p$, and given an integer $e\geq 0$, throughout this section we denote by $F_*^e R$ the following $(R,R)$-bimodule: for any $r,r_1,r_2\in R$:
\[
r_1 \cdot (F_*^e r)\cdot r_2 :=F_*^e (r_1^{p^e}rr_2).
\]
In other words, $F_*^e R$ is just $R$ as abelian group (actually, even as right $R$-module), but the reader will easily note that $R$ and $F_*^e R$ are completely different as left $R$-modules. Moreover, we also have to point out that one writes any element of $F_*^e R$ as $F_*^e r$, for some $r\in R$.

Now, we want to review the following notion (see \cite[Definition 2.4.2]{SmVdB97}) because it will play a key role during this section.



\begin{df}\label{strongly F-regular definition}
A reduced $F$-finite ring $R$ (i.e., a reduced ring of prime characteristic $p$ such that $R$ is a finitely generated $R^p$-module) is said to be \emph{strongly $F$-regular} if, for any $r\in R$ not in any minimal prime, there exists $e\in\N$ such that the map $\xymatrix@1{R\ar[r]& F_*^e R}$ sending $1$ to $F_*^e r$ splits as an $R$-module homomorphism.
\end{df}
As we have already pointed out during the Introduction, it is well-known that Zhang's question has an affirmative answer provided $R$ is a strongly $F$-regular domain; we want to review the proof for the convenience of the reader. We specially thank Luis Nu\~nez-Betancourt for pointing out to us the below proof.

\begin{teo}\label{annihilators of strongly F-regular rings}
Let $R$ be a strongly $F$-regular domain, and let $I\subseteq R$ be an ideal. Then, one has that $H_I^i (R)\neq 0$ if and only if $\left(0:_R H_I^i (R)\right)=0$.
\end{teo}

\begin{proof}
Assume that $\left(0:_R H_I^i (R)\right)\neq 0$ and pick a non-zero $r\in\left(0:_R H_I^i (R)\right)$. Since $R$ is strongly $F$-regular, there is some integer $e\geq 1$ such that the map $\xymatrix@1{R\ar[r]& F_*^e R}$ sending $1$ to $F_*^e r$ splits as an $R$-module homomorphism; let $\xymatrix@1{F_*^e R\ar[r]^-{\alpha}& R}$ a choice of splitting. Therefore, we have the following commutative triangle:
\[
\xymatrix{R\ar[dr]_-{\cdot F_*^e r}\ar@{=}[rr]& & R.\\ & F_*^e R\ar[ur]_-{\alpha}}
\]
After applying the functor $H_I^i$ to this triangle, we obtain the below one:
\[
\xymatrix{H_I^i (R)\ar[dr]_-{H_I^i(\cdot F_*^e r)}\ar@{=}[rr]& & H_I^i(R).\\ & H_I^i(F_*^e R)\ar[ur]_-{H_I^i(\alpha)}}
\]
However, it is easy to check that $r\in\left(0:_R H_I^i (R)\right)$ implies $H_I^i(\cdot F_*^e r)=0$, which implies that $H_I^i (R)=0$, just what we want to prove.
\end{proof}
Actually, a slightly more general statement works, with exactly the same proof of Theorem \ref{annihilators of strongly F-regular rings}.

\begin{teo}\label{annihilators of strongly F-regular rings in general}
Let $R$ be any strongly $F$-regular ring, and let $I\subseteq R$ be any ideal. Then, if $(0:_R H_I^i (R))\neq 0$, then $\left(0:_R H_I^i (R)\right)\cap R^{\circ}=\emptyset$, where $R^{\circ}$ denotes the set of nonzerodivisors of $R$.
\end{teo}

Our next aim is to partially recover Theorem \ref{annihilators of strongly F-regular rings} using our previous approach; before doing so, we need to review for the convenience of the reader the following concept (see \cite[Definition 3.1.1]{SmVdB97}).

\begin{df}\label{finite F-representation type}
Let $R$ be a commutative Noetherian ring of prime characteristic $p$. We say that $R$ has \emph{finite $F$-representation type} by finitely generated $R$-modules $M_1,\ldots ,M_s$ if, for any $e\geq 0$, the $R$-module $F_*^e R$ is isomorphic to a finite direct sum of the $R$-modules $M_1,\ldots ,M_s$, i.e., there are nonnegative integers $n_{e,1},\ldots ,n_{e,s}$ such that
\[
F_*^e R \cong\bigoplus_{j=1}^s M_j^{\oplus n_{e,j}}.
\]
Moreover, one simply says that $R$ has finite $F$-representation type provided there are finitely generated $R$-modules $M_1,\ldots ,M_s$ by which $R$ has finite $F$-representation type.
\end{df}

Many nice rings are of finite $F$-representation type; the interested reader may like to consult, among other places, \cite[Example 1.3]{TakTak08} and \cite{Shibuta11}.

In this way, we are finally in position to state and prove the following result, which partially recovers Theorem \ref{annihilators of strongly F-regular rings}.

\begin{teo}\label{answer to Zhang question}
Let $K$ be an $F$-finite field (i.e, $K$ is a finite dimensional $K^p$-vector space), and let $R$ be a commutative Noetherian ring, which is either a complete local ring of prime characteristic $p$ having $K$ as residue field, or a $\N$-graded ring, with $K$ as $0$th piece, such that $R$ is affine over $K$. Moreover, assume that $R$ is strongly $F$-regular and has finite $F$-representation type. Then, given any ideal $I$ of $R$ one has that $H_I^i (R)\neq 0$ if and only if $\left(0:_R H_I^i (R)\right)=0$.
\end{teo}

\begin{proof}
Under these assumptions, \cite[Theorem 4.2.1]{SmVdB97} ensures that the ring of $K$-linear differential operators $D_K (R)$ is simple. In this way, the result follows directly from Theorem \ref{simple implies faith} (with $A=R$, $B=D_K (R)$ and $M=H_I^i (R)$).
\end{proof}

\section{Annihilators of local cohomology via flat endomorphisms}\label{annihilators via flat endomorphisms}

The main goal of this section is to provide another alternative generalization of the Huneke-Koh result (see \cite[Lemma 2.2]{HunKoh91}). For this
reason, we exploit the idea used by Singh and Walther (see \cite[Remark 2.6]{SinghWalther2007}), in order to prove some vanishing results of local cohomology modules; it is worth noting that the same idea has been recently used by \`Alvarez Montaner to deduce certain results about Lyubeznik numbers (see \cite[Section 2]{AlvarezMontaner2015}).

Let $R$ be a commutative Noetherian ring with a flat endomorphism $\varphi:R
\rightarrow R$, and let $I$ be an ideal of $R$. We denote by $\varphi_* R$ the following $(R,R)$-bimodule: for any $r,r_1,r_2\in R$,
\[
r_1 \cdot(\varphi_* r)\cdot r_2 :=\varphi_* (\varphi (r_1)rr_2).
\]
Let $\Phi$ be the functor on the category of $R$-modules with $\Phi(M)=\varphi_* R \otimes_R M$. The iteration $\Phi^t$ is the functor
\[
\Phi^t(M)=\varphi_* R \otimes_R \Phi^{t-1}(M),\ \ t\geq 1,
\]
where $\Phi^0$ is interpreted as the identity functor; the reader will easily note that the flatness of $\varphi$ is equivalent to the exactness of $\Phi$. At once, one can realize that $\Phi^t(M)=\varphi_*^t R \otimes_R M$.

Let us notice that $\Phi(R) \cong R$ given by $\varphi_* r' \otimes r \mapsto\varphi(r)r'$. Furthermore, if $M$ and $N$ are $R$-modules, then (see \cite[(2.6.1)]{SinghWalther2007}) there are natural isomorphisms
\[
\Phi(\Ext^i_R(M,N))\cong \Ext^i_R(\Phi(M), \Phi(M)),\ \text{\ for\ all\ } i\geq 0.
\]
Assume, in addition, that the ideals $\{\varphi^t (I)R\}_{t\geq 0}$ form a descending chain cofinal with the chain $\{I^t\}_{t\geq 0}$; under these assumptions, one can easily check\footnote{There is a misprint in \cite[page 291, paragraph before (3)]{SinghWalther2007}; it says \emph{It follows that $H_{\fm}^i (R)\cong\Phi (H_{\fm}^i (R))$}, whereas it should say \emph{It follows that $H_{I}^i (R)\cong\Phi (H_{I}^i (R))$}.} (see \cite[page 291]{SinghWalther2007} for details) that
\[
\Phi(H^{i}_{I}(R)) \cong H^{i}_{\varphi(I)}(R) \cong H^{i}_{I}(R),\ \text{\ for\ all\ } i\geq 0.
\]

\begin{rk}
The Frobenius endomorphism of a ring of positive characteristic and, given any field $K$ and any integer $t\geq 2$, the $K$-linear endomorphism $\varphi(x_i) = x^t_i$ of $K[x_1,\ldots,x_n]$ are the prototypical examples of flat endomorphisms. The interested reader on flat endomorphisms may like to consult \cite[Chapter 1]{Miasnikov14} and the references therein for additional information.
\end{rk}
In this way, we are ready for producing the promised generalization of the Huneke-Koh result, which turns out to be the main statement of this section.

\begin{teo} \label{flat endomorphisms}
Let $R$ be a commutative Noetherian ring, which is either a domain, or local. Moreover, suppose that there is $\varphi:R \rightarrow R$ a flat endomorphism such that, for any proper ideal $J$ of $R$, $\varphi^t (J)R\subseteq J^t$ for any $t\geq 0$, and $\{\varphi^t(J)R\}_{t\geq 0}$ is a decreasing chain of ideals cofinal with the chain $\{J^t\}_{t\geq 0}$. Then, one has that $H^{j}_{I}(R)\neq 0$ if and only if $\left(0:_R H^{j}_{I}(R)\right)=0.$
\end{teo}

\begin{proof}
Indeed, assume that $H^{j}_{I}(R)\neq 0$. Let $x$ be a non-zero element of $H^{j}_{I}(R)$ and let $J = \left(0:_R x\right)$. Then, there is an injection $R/J \hookrightarrow H^{j}_{I}(R).$ Fix $t\geq 0$; since $\Phi^t$ is exact, after applying $\Phi^t$ to the previous injection, we get once again an injection
\[
\Phi^t(R/J) \cong R/\varphi^t (J)R \rightarrow \Phi^t(H^{j}_{I}(R)) \cong
H^{j}_{I}(R).
\]
Hence, Krull's Intersection Theorem (see \cite[Corollary 5.4]{Eisenbud1995}) implies that
\[
\left(0:_R H^{j}_{I}(R)\right) \subset \bigcap_{t \geq 0} \varphi^t (J)R\subseteq \bigcap_{t \geq 0} J^t=0,
\]
just what we wanted to show.
\end{proof}

\section{Annihilators of local cohomology modules on normal reduced rings}\label{annihilators of normal rings}
The purpose of this section is to provide several results which ensure the vanishing of annihilators of local cohomology modules on normal reduced rings.

\begin{prop}\label{annihilator vanishing from primary components}
Let $R$ be a Noetherian normal reduced ring with $\left(0:_R H^j_I(R/\fp)\right)=0 $ for at least one minimal prime ideal $\fp$ of $R$. Then  $\left(0:_R H^j_I(R)\right)=0$ for a given integer $j$.
\end{prop}

\begin{proof}
Let $\fp_1, \ldots, \fp_r$ be all the minimal prime ideals of $R$. It follows from \cite[Corollary 2.1.13]{HunekeSwanson2006} that $R \cong \oplus^r_{i=1} R/\fp_i,\ 1 \leq i \leq r,$ where each $R/\fp_i$ is an integrally closed domain. Since local cohomology commutes with direct sums (see \cite[3.4.10]{BroSha}), one has that
\[
 H^{j}_{I}(R) \cong \oplus^r_{i=1}
H^{j}_{I}(R/\fp_i),\ 1 \leq i \leq r.
\]
In this way, from the previous isomorphism it follows that
\[
\left(0:_R H^{j}_{I}(R)\right) = \bigcap^r_{i=1} \left(0:_R H^{j}_{I}(R/\fp_i)\right),
\]
and therefore the lefthandside ideal is zero because the righthandside one is so by the assumptions.
\end{proof}

Next result may be regarded as a byproduct of Theorem \ref{flat endomorphisms} and Proposition \ref{annihilator vanishing from primary components}.

\begin{cor}\label{byproduct of results}
Let $R$ be a Noetherian normal reduced ring, and let $\fp_1,\ldots ,\fp_r$ be its minimal primes. Moreover, suppose that, for some $1\leq i\leq r$, there exists a flat endomorphism $\xymatrix@1{R/\fp_i\ar[r]^-{\varphi_i}& R/\fp_i}$ satisfying assumptions (a) and (b) of Theorem \ref{flat endomorphisms}. Then, $H^{j}_I(R/\fp_i)\neq 0$ if and only if $(0:_{R/\fp_i} H^j_I(R/\fp_i))=0$ for $1\leq i\leq r$.
\end{cor}

\begin{proof}
Once again, since
\[
R\cong\bigoplus_{i=1}^r R/\fp_i
\]
it follows, as in the proof of Proposition \ref{annihilator vanishing from primary components}, that
\[
 H^{j}_{I}(R) \cong \oplus^r_{i=1}
H^{j}_{I}(R/\fp_i),\ 1 \leq i \leq r.
\]
Therefore, the result follows directly from Theorem \ref{flat endomorphisms} and the proof is therefore completed.
\end{proof}



Next result, which turns out to be the main result of this section, should be compared with \cite[Theorem 4.4]{Lyn12}; it may be regarded as a mild improvement of \cite[Lemma 3.6]{AtSeNa14}.

\begin{teo}\label{main result of section on normal rings}
Let $I \subset R$ be an ideal of a Noetherian normal reduced ring of dimension $d$, and suppose that $c:=\cdim (I)=d-1$. Then $(0:_{R} H^c_I(R))=0$.
\end{teo}

\begin{proof}
  Let $\fp_1, \ldots, \fp_r$ be all the minimal prime ideals of $R$.
It follows from the proof of Proposition \ref{annihilator vanishing from primary components} that $ (0:_R H^c_I(R)) =
\bigcap^r_{i=1} (0:_R H^c_I(R/\fp_i))$. Note that  each $R/\fp_i$ is an integrally closed domain, hence every localization of $R/\fp_i$ is integrally closed domain for every prime ideal of $R/\fp_i$, i.e. $R/\fp_i$ is a  normal domain. As 
 at least of one the modules
$H^c_I(R/\fp_i)\neq 0$ and $H^{c+1}_I(R/\fp_i)  \cong H^{c+1}_I(R)
\otimes_R R/\fp_i= 0$. Thus, without loss of generality we may assume that $R$ is a Noetherian normal domain. Now, the claim follows from \cite[Lemma 3.6]{AtSeNa14}.
\end{proof}

Finally, we want to emphasize that Zhang's question is, in general, not related with the question of finiteness of associated primes of local cohomology modules; this will be illustrated with the below:

\begin{ex}\label{no connection with finiteness of associated primes}
Let $K$ be a field, and set
\[
R:=\frac{K[t,u,v,x,y]}{\left(u^2x^2+tuvxy+v^2y^2\right)}.
\]
It is known (see \cite[Theorem 4.1 and Remark]{SinghSwanson2004}) that $R$ is a $4$-dimensional domain such that
\[
\# \Ass_R \left(H_{(x,y)}^2 (R)\right)=\infty .
\]
However, Theorem \ref{annihilators of strongly F-regular rings} ensures that $\left(0:_R H_{(x,y)}^2 (R)\right)=0$. This particular example shows that, in general, the vanishing of the annihilator of a local cohomology module does not imply the finiteness of its set of associated primes.
\end{ex}


\section{Annihilators of local cohomology and Cohen-Macaulay rings}\label{annihilators and Macaulay rings}
The purpose of this section is to study the following question, originally raised by Lynch in her thesis:

\begin{quo}[Lynch]\label{Lynch question on Macaulay rings}
Let $R$ be a Cohen-Macaulay local ring, and let $I$ be any ideal of $R$. Is it true that $H_I^i (R)\neq 0$ if and only if $\left(0:_R H_I^i (R)\right)=0$?
\end{quo}
As application of Theorem \ref{annihilators over regular rings: main result}, we shall see in this section that, if $R$ contains a field and $I$ is a monomial ideal (see Definition \ref{monomial ideals in regular sequences}), then Question \ref{Lynch question on Macaulay rings} can be reduced to the case of a polynomial ring; before doing so, we need to review the following preliminary notions (see, among other places, \cite[Notation 1]{KiyekStuckrad2003}):

\begin{df}\label{monomial ideals in regular sequences}
Let $R$ be a commutative Noetherian ring, and let $x_1,\ldots ,x_d$ be an $R$-regular sequence contained in the Jacobson radical (i.e., in the intersection of all maximal ideals) of $R$. On one hand, an element $r$ of $R$ is called a \emph{monomial} (with respect to $x_1,\ldots ,x_d$) provided $r=\mathbf{x}^{\alpha}:=x_1^{a_1}\cdots x_d^{a_d}$, for some $\alpha =(a_1,\ldots ,a_d)\in\N^d.$ On the other hand, an ideal $I$ of $R$ is called a \emph{monomial ideal} (with respect to $x_1,\ldots ,x_d$) if it is generated by monomials.
\end{df}
In this way, we are ready to state and prove the main result of this section; namely:

\begin{teo}\label{some answer to Lynch question on Macaulay rings}
Let $R$ be a Cohen-Macaulay local ring containing a field $K$, and let $I$ be a monomial ideal with respect to some system of parameters $x_1,\ldots ,x_d$ for $R$. Then, $H_I^i (R)\neq 0$ if and only if $\left(0:_A H_J^i (A)\right)=0$, where $A:=K[X_1,\ldots ,X_d],$ and $J\subseteq A$ is the
preimage of $I$ under the $K$--algebra homomorphism $\xymatrix@1{A:=K[X_1,\ldots ,X_d]\ar[r]^-{\phi}& R}$ mapping each $X_j$ to $x_j$.
\end{teo}

\begin{proof}
Let $x_1,\ldots ,x_d$ be a system of parameters for $R$ such that $I$ is a monomial ideal with respect to $x_1,\ldots ,x_d$; by \cite[Proposition 1]{Hartshorne1966sequences}, the $K$-algebra homomorphism $\xymatrix@1{A:=K[X_1,\ldots ,X_d]\ar[r]^-{\phi}& R}$ mapping each $X_j$ to $x_j$ is flat. Let $J\subseteq A$ as the preimage of $I$ under $\phi$; since $I$ is a monomial ideal of $R$, $J$ is a (usual) monomial ideal of $A$. Moreover, the Independence Theorem (see \cite[4.2.1]{BroSha}) ensures that $H_I^i (R)\cong H_J^i (A)$ as $A$-modules. In this way, combining the previous information jointly with Theorem \ref{annihilators over regular rings: main result} it follows that $H_I^i (R)\neq 0$ if and only if $H_J^i (A)\neq 0$ if and only if $\left(0:_A H_J^i (A)\right)=0;$ the proof is therefore completed.
\end{proof}





\section*{Acknowledgments}
The authors would like to thank Josep \`Alvarez Montaner, Pilar Bayer, Gennady Lyubeznik, Reza Naghipour, Luis Narv\'aez Macarro, Luis N\'u\~{n}ez-Betancourt, Nicholas Switala and Santiago Zarzuela for many helpful comments concerning an early draft of this paper.

\bibliographystyle{alpha}
\bibliography{AFBoixReferences}

\end{document}